\RequirePackage{fix-cm}
\documentclass[smallextended]{svjour3}       
\smartqed  
\usepackage{graphicx}
\usepackage{amsmath,amssymb}
\newcommand\Ker{\mathop{\rm Ker}}
\newcommand\im{\mathop{\rm Im}}
\newcommand\dist{\mathop{\rm dist}}

\newtheorem{thm}{Theorem}

\newtheorem{defn}{Definition}
\begin{document}

\title{Implicit function and tangent cone theorems for singular inclusions and applications to nonlinear programming
}


\author{Agnieszka Prusi\'nska \and
Ewa Bednarczuk         \and
        Alexey Tret'yakov
}


\institute{A. Prusi\'nska \at
               Siedlce University, Faculty of Science, 3 Maja 54, 08-110 Siedlce, Poland\\
               Tel.: +48-25-6431076\\
               \email{aprus@uph.edu.pl}\\
               ORCID ID: 0000-0002-6091-6884
              \and
E. Bednarczuk \at
              System Research Institute, Polish Academy of Sciences, Newelska 6, 01-447 Warsaw, Poland
           \and
              A. \ Tret'yakov \at
              Dorodnitsyn Computing Center, Russian Academy of Sciences, Vavilova 40, 119991 Moscow, Russia\\
}

\date{Received: date / Accepted: date}

\maketitle

\begin{abstract}
The paper is devoted to the implicit function theorem involving singular mappings. We also discuss the form of the tangent cone to the solution set of the generalized equations in singular case and give some examples of applications to nonlinear programming and complementarity problems.
\keywords{Implicit function theorem \and Multifunction \and Tangent cone \and Singularity \and Generalized equation \and $p$-regularity}
\subclass{46N10 \and 47J07 \and 47J22}
\end{abstract}

\section{Introduction}
\label{intro}
The paper deals with so-called generalized equations, that is inclusions of the form
\begin{equation}\label{eq1}
0\in f(x,y)+Q(y)
\end{equation}
where $f:X\times Y\mapsto Y^{\ast}$, $Y$ -- Banach space, $Y^{\ast}$ dual space to $Y$, $X$ -- normed space and  $Q:Y\rightrightarrows Y^{\ast}$ is a set valued mapping.

Generalized equations provide a useful tool for the analysis of  complementarity problems, first order optimality conditions in mathematical programming, equilibrium problems, and many aspects of nonlinear analysis.
Recently metric regularity and Lipschizian stability of solutions to \eqref{eq1} have been studied in \cite{ArMo10}.
 In \cite{CiDoKrVe16} the generalized equation \eqref{eq1} has been accompanied by differential equation and is called differential generalized equation (DGE). The authors provided characterization of metric regularity of  control system coupled with a differential generalized equation, and characterize metric regularity of (DGE) in the finite dimensional case (see also \cite{Gw07,Gw13,PaSt08} for other results related to existence and stability issues). For generalized equations of the form
 $$0\in x- G(y),$$
 where $G:X\times Y \rightrightarrows Z$ and $X,Y,Z$ -- Banach spaces the inverse mapping theorems and higher order metric regularity have been discussed in \cite{Fr89,Fr90,FrQu12}.

  We investigate the problem \eqref{eq1} with the set valued mapping $Q$ of the form of the normal cone $Q=N_C(y)$, where $C\subset Y$ is a convex set, i.e.
 \begin{equation}\label{eq1a}
 0\in f(x,y)+N_C(y),
 \end{equation}
 see e.g. \cite{KlKu02,RoWe98}. Here $N_C(y): Y\rightrightarrows Y^{\ast}$ is defined as
 \begin{equation}\label{conv}
 N_C(y):=\left\{\begin{array}{lll}
 \{z\in Y^{\ast}: \langle z,c-y\rangle\leq 0,\; \forall c\in C\}, & \hbox{ if } & y\in C\\
 \emptyset, & \hbox{ if } & y\notin C.
 \end{array}
 \right.
 \end{equation}

 In the present paper we prove implicit function theorem for the problem \eqref{eq1a}. In \cite{AvMa12}, an implicit function theorem is provided for the problem $f(x,y)\in A$, where $f:X\times Y\mapsto Z $ is a differentiable mapping with $f'_y$ being onto, $X$ topological space, $Y$, $Z$ -- Banach spaces and $A\subset Z$, see also \cite{Iz14,Rob80}.

The approach to implicit function theorem for generalized equation \eqref{eq1a} we propose differs from the results of the above mentioned papers in that we focus on singular mappings, i.e. we do not assume surjectivity of the derivative of the mapping $f$. Such problems \eqref{eq1a} are called singular inclusion problems.

Let $(x_0,y_0)$ be the solution to the inclusion problem \eqref{eq1a}, i.e.
\begin{equation}\label{eq2}
0\in f(x_0,y_0)+N_C(y_0).
\end{equation}

We use $p$-regularity theory \cite{BrTrMa08} in deriving our implicit function theorem for \eqref{eq1a} which is an efficient tool to deal with differentiable mappings when the first derivative of $f$ at $(x_0,y_0)$ is not surjective (singular, degenerate) and so the strong regularity condition by Robinson (see \cite{Rob80}) is not applicable.

More precisely, we investigate the problem of existence of a locally defined mapping $\varphi : X\rightarrow Y$, $y=\varphi(x)$ which is a solution of \eqref{eq1a} near a given solution $(x_0,y_0)$, that is $0\in f(x,\varphi(x))+N_C(\varphi(x))$ for $x$ close to $x_{0}$ and $y_0=\varphi(x_0)$.

The classical implicit function theorem says that when a continuously differentiable function $f(x,y)$ vanishes at a point $(x_0,y_0)$ with $f'_y(x_0,y_0)$ nonsingular (surjective), the equation $f(x,y)=0$ can be solved for $y$ in terms of $x$ in a neighborhood of $(x_0,y_0)$. This theorem has been extended in various directions, e.g. to Banach spaces \cite{Tr87,TrMa03}, to multivalued mappings \cite{AvMa12,Fr89,Fr90,Iz14,Rob80}, to nonsmooth functions \cite{Iz14,Rob81}, etc.

The results we present can be applied to parametric problems. There are numerous theorems concerning the solution existence of the problem with small parameter. Some of them deal with the problem of solution existence of the equation $f(x,y)=0$,  where the mapping $f$ is singular, e.g. \cite{BrEvTr06,BrEvTre06,PrTr08,PrTr11,Tr10}. This analysis was based on the constructions of $p$-regularity theory that has been developed for the last forty years. The main constructions of this theory are described e.g. in \cite{BrTr03}--\cite{BrTr07} or in \cite{Tr87}--\cite{Tr10}.

The organization of the paper is as follows. In section 2 we formulate our main result which is Theorem \ref{th1} providing conditions for the existence of the implicit function $\varphi$ for the singular inclusion problem \eqref{eq1a}.
In section 3 we prove Lusternik-type theorem for \eqref{eq1a}.

\section{Implicit Function Theorem for singular inclusions}
\label{sec:1}
Let $p\geq 2$ be a natural number and let $B:X^p\rightarrow Y$ be a continuous symmetric $p$-multilinear mapping. By $B[\cdot]^p: X\rightarrow Y$ we mean the $p$-form associated to $B$ and define it as follows
$$B[x]^p:=B(x,x,\ldots,x), \qquad x\in X.$$

For any $h \in Y$ and $f\in \mathcal{C}^{p+1}(X\times Y,Y^{\ast})$ we define a set-valued mapping $L_h: Y\rightrightarrows Y^{\ast}$,
$$L_h(y):=\frac{1}{(p-1)!} f_y^{(p)}(x_0,y_0)[h]^{p-1}[h+y]+N_C(h+y).$$
Recall the Hausdorff distance between any sets $S_1$ and $S_2$,
$$H(S_1,S_2):=\max \left\{\sup_{x\in S_1}\dist(x,S_2),\sup_{y\in S_2}\dist(y,S_1)\right\}.$$

Without loss of generality assume that $x_0= 0$ and $y_0= 0$. Denote $U_{\gamma}(0)$, $V_{\gamma}(0)$ sufficiently small neighborhoods of $0$ in $X$ and $Y$, respectively. Consider completely degenerate case up to the order $p$, that is assume that $f_y^{(k)}(0,0)\equiv 0$ for $k=1,\ldots,p-1.$

We prove the following theorem.
\begin{thm}\label{th1}
  Let $f\in \mathcal{C}^{p+1}(X\times Y,Y^{\ast}).$ Suppose that \eqref{eq2} is satisfied at $(x_0,y_0)=(0,0)$ and
  \begin{equation}\label{eq3}
    f_y^{(k)}(0,0)=0, \quad k=1,2,\ldots, p-1.
    \end{equation}
  Assume the following conditions hold.
    \begin{enumerate}
      \item[$1^{\circ}$] Banach condition:

\vskip3pt
\noindent
      for any $x\in U_{\gamma}(0)$, such that $f(x,0)\neq 0$ and $\gamma>0$ is sufficiently small, there exists $h(x)\in Y$, $h\neq 0$ such that
\begin{equation}\label{eq4}
  -f(x,0)\in \frac{1}{(p-1)!}f_y^{(p)}(0,0)[h(x)]^p+N_C(h(x)),
\end{equation}
and
$\|h(x)\|\leq c\cdot \|f(x,0)\|^{1/p}$, where $c>0$ is independent constant,

\vskip6pt

      \item[$2^{\circ}$] Strong $p$-regularity condition at the point $0$ along $h=h(x)$, $h\in Y$, i.e.

\vskip3pt
      \begin{equation}\label{eq5}
        H(L_h^{-1}(z_1),L_h^{-1}(z_2))\leq \frac{c}{\|h\|^{p-1}}\|z_1-z_2\|, \quad \forall z_1, z_2 \in Y^{\ast},
      \end{equation}

\vskip6pt
      \item[$3^{\circ}$] $p$-factor approximation condition, i.e.

      \vskip3pt
      \noindent
      \begin{eqnarray}\label{eq6}
        \nonumber \left\|f(x,y_1)-f(x,y_2)-\frac{1}{p!}f_y^{(p)}(0,0)[y_1]^p+\frac{1}{p!}f_y^{(p)}(0,0)[y_2]^p\right\| &\leq&  \\
        \leq \delta \left(\|y_1\|^{p-1}+\|y_2\|^{p-1}\right)\|y_1-y_2\|& &
      \end{eqnarray}
         for $x\in U_{\gamma}(0),$ $y_1$, $y_2\in V_{\gamma}(0)$ and $\delta>0$  sufficiently small.
    \end{enumerate}
  Then for sufficiently small $\varepsilon>0$ there exist a neighborhood $U_{\varepsilon}(0)\subset X$ and  a mapping $\varphi(x): U_{\varepsilon}(0)\rightarrow Y$ such that for any $x\in U_{\varepsilon}(0)$ the mapping $\varphi(x)$ is a solution of the inclusion \eqref{eq1a}, i.e.
  \begin{equation}\label{eq7}
    0\in f(x,\varphi(x))+N_C(\varphi(x)),
  \end{equation}
  and
  \begin{equation}\label{eq7'}
    \|\varphi(x)\|\leq m\cdot \|f(x,0)\|^{1/p},
  \end{equation}
  where $m>0$ is independent constant.
\end{thm}

\noindent
{\bf Remark.}
In the case when $f(x,0)=0$ we can take the mapping $\varphi(x)=0$ and Banach condition is trivial.

\vskip6pt

Before we prove this theorem we give two examples and for the convenience of the Reader we recall Robinson's strong regularity condition \cite{Rob80}.
\begin{defn}\label{rob}
	Let $f\in \mathcal{C}^1(X\times Y,Y^{\ast})$ and $(x_0,y_0)$ be a solution of \eqref{eq1a} and
	$$Ty:=f(x_0,y_0)+f'_y(x_0,y_0)(y-y_0)+N_C(y).$$
	We say that \eqref{eq1a} is strongly regular at $(x_0,y_0)$ with associated Lipschitz constant $\lambda$ if there exist neighborhoods $U$ of the origin in $Y^{\ast}$ and $V$ of $y_0$ such that the restriction to $U$ of $T^{-1}\cap V$ is a single valued function from $U$ to $V$ which is Lipschitzian on $U$ with modulus $\lambda$.
\end{defn}

\noindent
\textbf{Example 1.}

Consider the nonlinear complementarity problem: to solve the system
\begin{equation} \label{eq8}
\begin{split}
f_1(x,y) & = y_1^2-y_2^2-x_1\geq 0 \\
 f_2(x,y) & = y_1\cdot y_2-x_2 \geq 0\\
 y\geq 0,&\quad \langle f(x,y),y \rangle=0
\end{split}
\end{equation}
where
$$x=(x_1,x_2)^T\in \mathbb{R}^2, \; y=(y_1,y_2)^T\in \mathbb{R}_+^2, \;  f(x,y)=(f_1(x,y),f_2(x,y))^T, $$
$f:\mathbb{R}^2\times \mathbb{R}_+^2 \to \mathbb{R}^2$
and $x$ represents a small perturbation parameter. This problem is equivalent to the following generalized equation (see \cite{Rob80})
\begin{equation}\label{eq9}
  0\in f(x,y)+N_{\mathbb{R}_+^2}(y)
\end{equation}
and to analyze nonlinear complementarity problem \eqref{eq8} we can investigate inclusion \eqref{eq9} and apply Theorem \ref{th1}. It is obvious that the strong regularity condition (see Definition \ref{rob}) fails at $(0,0)^T$ (since $f'_y(0,0)\equiv 0$) and $T^{-1}$ is multivalued mapping, where $Ty:=f'_y(0,0)y+N_{\mathbb{R}_+^2}(y).$

On the other hand, it turns out that all assumptions of Theorem \ref{th1} are fulfilled for $p=2$.

In this example the Banach condition $1^{\circ}$ takes the form $$-f(x,0)\in f_y''(0,0)[h(x)]^2+N_{\mathbb{R}_+^2}((h_1,h_2)),$$ that is
\begin{equation}\label{eq10}
  -\left(
     \begin{array}{c}
       x_1 \\
       x_2 \\
     \end{array}
   \right)\in \left(
                \begin{array}{c}
                  2h_1^2-2h_2^2 \\
                  2h_1 h_2 \\
                \end{array}
              \right)+N_{\mathbb{R}_+^2}((h_1,h_2))
\end{equation}
holds since $h_1\neq 0$ and $h_2\neq 0$ and hence $N_{\mathbb{R}_+^2}(h)= \{0\}$. This yields that the solution $h(x)$ of \eqref{eq10} can be found as a solution of the following equation
\begin{equation*}
  -\left(
     \begin{array}{c}
       x_1 \\
       x_2 \\
     \end{array}
   \right)= \left(
                \begin{array}{c}
                  2h_1^2-2h_2^2 \\
                  2h_1 h_2 \\
                \end{array}
              \right)
\end{equation*}
and we obtain $\|h(x)\|\leq c \|f(x,0)\|^{1/2}$, $c>0$.

The second assumption, strong $p$-regularity $2^{\circ}$ holds since $f(x,y)$ is $2$-regular (see e.g. in \cite{BrTr07}) at $(0,0)^T$ with respect to $y$ along any $h\in Y$, such that $h_1\neq 0$ and $h_2\neq 0$, that is $\im f''(0,0)h=\mathbb{R}^2$ and $N_{\mathbb{R}_+^2}(\cdot)=\{0\}$.

The $p$-factor approximation condition $3^{\circ}$ is immediately satisfied due to the form of the mapping $f(x,y)$. This means that all conditions of theorem \ref{th1} are fulfilled and therefore for small $x$ there exists a mapping $\varphi(x)$ which acts from a neighborhood of $0\in \mathbb{R}^2$ into $\mathbb{R}^2$ and for which \eqref{eq7'} is satisfied.

\vskip6pt

In the next example we illustrate the application of Theorem \ref{th1} to the standard nonlinear programming problem
\begin{equation} \label{eq11}
\begin{split}
 & \min \xi(x) \\
 \hbox{subject to } & g(y)\leq 0,
 \end{split}
\end{equation}
where $g=(g_1,\ldots,g_m)^T$, $\mathcal{L}(y,\lambda)=\xi(y)+\langle \lambda,g(y)\rangle,$ $\lambda=(\lambda_1,\ldots,\lambda_m)^T$. The Karush-Kuhn-Tucker optimality conditions (KKT) for \eqref{eq11} are
\begin{equation}\label{eq12}
  \mathcal{L}'_y=0, \; g(y)\leq 0, \; \lambda\geq 0, \; \langle\lambda, g(y)\rangle=0.
\end{equation}
This conditions can be written as the generalized equation
\begin{equation}\label{eq13}
  0\in \left(
     \begin{array}{c}
       \mathcal{L}'_y(y,\lambda) \\
       -g(y) \\
     \end{array}
   \right)+ N_{\mathbb{R}^n\times\mathbb{R}_+^m}\left(
                \begin{array}{c}
                  y \\
                  \lambda \\
                \end{array}
              \right)
\end{equation}
and we verify the assumptions of the Theorem \ref{th1}  at the solution $(y_0,\lambda_0)$ for \eqref{eq13}.

\vskip6pt

\noindent
\textbf{Example 2.}

Consider the following nonlinear programming problem
\begin{equation} \label{eq14}
\begin{split}
 & \min y_1^4-y_2^4-xy_1 \\
 \hbox{subject to } & \\
 & y_1^3-2y_2^3\leq 0\\
 & y_1^3+2y_2^3\leq 0
 \end{split}
\end{equation}
where $x$ represents a small perturbation parameter. For $x_0=0$ a solution of \eqref{eq14} is $y_0=0$ and
$\mathcal{L}(y,\lambda,x)=y_1^4-y_2^4-xy_1+\lambda_1(y_1^3-2y_2^3)+\lambda_2 (y_1^3+2y_2^3)$.

Then Karush-Kuhn-Tucker optimality conditions are as follows
$$4y_1^3-x+3\lambda_1y_1^3+3\lambda_2y_1^2=0$$
$$-4y_2^3-6\lambda_1  y_2^2+6\lambda_2y_2^2=0$$
$$y_1^3-2y_2^3\leq 0, \quad y_1^3+2y_2^3\leq 0, \quad \langle\lambda,g(y)\rangle=0$$
where
$$\lambda=(\lambda_1,\lambda_2)^T, \; y=(y_1,y_2)^T, \; g(y)=(g_1(y),g_2(y))^T,$$
$$g_1(y)=y_1^3-2y_2^3, \; g_2(y)=-y_1^3+2y_2^3.$$
Consider the case $\lambda_0=(0,0)^T$.
Generalized equation \eqref{eq1a} for the problem \eqref{eq14} is
\begin{equation}\label{eq15}
  0\in \left(
     \begin{array}{c}
       4y_1^3-x+3\lambda_1y_1^2+3\lambda_2y_1^2 \\
       -4 y_2^3-6\lambda_1y_2^2+6\lambda_2y_2^2\\
       y_1^3-2y_2^3\\
       y_1^3+2y_2^3\\
     \end{array}
   \right)+ N_{\mathbb{R}^2\times\mathbb{R}_+^2}\left(
                \begin{array}{c}
                  y_1 \\
                  y_2\\
                  \lambda_1 \\
                  \lambda_2\\
                \end{array}
              \right).
\end{equation}
It is obvious that \eqref{eq15} is not strong regular in the sense by Robinson at the solution point $y_0=(0,0)^T,$ $\lambda_0=(0,0)^T,$ $x_0=0$. However, one can verify that all assumptions of Theorem \ref{th1} are fulfilled for $p=3$ with $h=(h_{y_1},h_{y_2},0,h_{\lambda_2})$, since in this case we can take $h_{y_1}\neq 0$, $h_{y_2}\neq 0$, $h_{\lambda_2}\neq 0$ for $x\neq 0$ and normal cone operator $N_{\mathbb{R}^2\times\mathbb{R}_+^2}\left(
                \begin{array}{c}
                  y_1 \\
                  y_2\\
                  \lambda_1 \\
                  \lambda_2\\
                \end{array}
              \right)$ has the following form $\left(
                \begin{array}{c}
                 0 \\
                  0\\
                  -\alpha \\
                  0\\
                \end{array}
              \right)$, $\alpha\in\mathbb{R}_+$.

 This means that for small perturbation $x$ there exists a mapping $\varphi(x)=(y(x),\lambda(x))^T$ such that inclusion \eqref{eq15} (or \eqref{eq13}) holds and hence KKT conditons \eqref{eq12} holds as well and the estimation for $\varphi(x)$ is as follows
$$\|\varphi(x)\|=\|y(x)\|+\|\lambda(x)\|\leq m\cdot \|\mathcal{L}_y'(0,0,x)\|^{1/3}\leq \bar{m} \|x\|^{1/3}$$
where $m>0$, $\bar{m}>0$ are independent constants.

\vskip6pt

The following theorem is essential in the proof of Theorem \ref{th1} (see \cite{IoTih74}).
\begin{thm}[Contraction multimapping principle, CMP]
	Let $Z$ be a complete metric space with distance $\rho$ and $z_0\in Z$. Assume
	that we are given a multimapping
	\[\Phi : U_{\varepsilon}(z_{0})\rightrightarrows Z,\] on a ball
	\(U_{\varepsilon}(z_{0})=\left\{z:
	\rho(z,z_{0})<\varepsilon\right\} \;\;(\varepsilon>0)\) where the
	sets $\Phi(z)$ are non-empty and closed for any \(z\in
	U_{\varepsilon}(z_{0}).\) Further, assume that there exists a
	number $\theta,\; 0<\theta<1$ such that
	\begin{enumerate}
		\item[1)] $H(\Phi(z_{1}),\Phi(z_{2}))\leq \theta
		\rho(z_{1},z_{2})$ for any $z_{1},z_{2}\in U_{\varepsilon}(z_{0})$
		\item[2)] $\rho(z_{0},\Phi(z_{0}))<(1-\theta) \varepsilon.$
	\end{enumerate}
	Then, for every number $\varepsilon_{1}$ which satisfies the
	inequality
	\[\rho(z_{0},\Phi(z_{0}))<\varepsilon_{1}<(1-\theta)\varepsilon,\]
	there exists \(z\in
	B_{\varepsilon_{1}/(1-\theta)}(z_{0})=\left\{\omega:
	\rho(\omega,z_{0})\leq\varepsilon_{1}/(1-\theta)\right\}\) such
	that
	$$
	z\in \Phi (z).
	$$
\end{thm}

\begin{proof}(of Theorem 1)

\noindent
Suppose that $x\in U_{\gamma}(0)$, $y_1,y_2\in V_{\gamma}(0)$, for sufficiently small $\gamma>0$. Moreover, assume that there exists $h(x)\in Y$, $h\neq 0$ such that Banach condition holds true.

Let us define a mapping $r: U_{\gamma}(0)\times Y\rightarrow W_{\varepsilon}(0)$, where $\varepsilon> 0$ sufficiently small and $W_{\varepsilon}(0) \subset Y^{\ast}$, as follows
  \begin{equation}\label{eq16}
  r(x,h+y):=\frac{1}{(p-1)!}f_y^{(p)}(0,0)[h]^{p-1}[h+y]-f(x,h+y).
  \end{equation}
  Now, for $y=0$
  \begin{equation}\label{eq17}
    L_h(0)=\frac{1}{(p-1)!}f_y^{(p)}(0,0)[h]^{p}+N_C(h).
  \end{equation}
  Then, using a right inverse of $L_h$ we obtain
  \begin{equation}\label{eq18}
    0\in L_h^{-1}\left(\frac{1}{(p-1)!}f_y^{(p)}(0,0)[h]^{p}+N_C(h)\right).
  \end{equation}
  Introduce an auxiliary mapping $\Phi: U_{\gamma}(0)\times V_{\gamma}(0)\rightarrow Y$,
  \begin{equation}\label{eq19}
    \Phi(x,y):=L_h^{-1}(r(x,h+y)).
  \end{equation}
  We show that there exists $y=y(x)\in \Phi(x, y(x))$ or, in other words,\\ $0\in N_C(h+y(x))+f(x,h+y(x)).$
  For this purpose we check the assumptions of CMP.

  From the assumptions 2$^{\circ}$ and 3$^{\circ}$ of Theorem \ref{th1} we obtain
  \begin{enumerate}
    \item[1)] $H(\Phi(x, y_1), \Phi(x,y_2))\leq \!\!\!\!\!\!\!^{^{2^{\circ}}} \frac{c}{\|h\|^{p-1}}\|r(x,h+y_1)-r(x,h+y_2)\|=\frac{c}{\|h\|^{p-1}}\cdot$

        \vskip4pt

        $\cdot\|\frac{1}{(p-1)!}f_y^{(p)}(0,0)[h]^{p-1}[y_1-y_2]-f(x,h+y_1)+f(x,h+y_2)\|\leq$

\vskip4pt

    $\leq \!\!\!\!\!\!^{^{3^{\circ}}}\frac{c\cdot \delta(\|h\|^{p-1}+\|h\|^{p-1})}{\|h\|^{p-1}}\|y_1-y_2\|\leq 2 c\cdot \delta \|y_1-y_2\|$.

\vskip4pt

    \hskip-0.5cm
    Moreover, since $ 0\in L_h^{-1}\left(-f(x,0)\right)$ then

    \vskip4pt

    \item[2)] $H(\Phi(x,0), 0)\leq H\left(L_h^{-1}(r(x,h)),L_h^{-1}(-f(x,0))\right)\leq$

\vskip4pt

        $\leq \!\!\!\!\!\!^{^{2^{\circ}}} \frac{c}{\|h\|^{p-1}}\|r(x,h)+f(x,0)\|=$

\vskip4pt

        $=\frac{c}{\|h\|^{p-1}}\left\|\frac{1}{(p-1)!}f_y^{(p)}(0,0)[h]^{p}-f(x,h)+f(x,0)\right\|\leq$

\vskip4pt

        $\leq \!\!\!\!\!\!^{^{3^{\circ}}} \frac{c\cdot\delta}{\|h\|^{p-1}}\|h\|^{p-1}\|h\|\leq c \delta\|h\|$.
  \end{enumerate}

\noindent
It means that all assumptions of CMP hold.
Hence there exists $y=y(x)$ such that $y(x)\in L^{-1}_h(r(x,h+y(x)))$, or in other words $$ 0\in f(x,h+y(x))+N_C(h+y(x)).$$

Let $\varphi(x):=h+y(x)$. We finish the proof of the Theorem \ref{th1}  by obtaining the following estimation
$$\|y(x)\|=o\left(\|f(x,0)\|^{1/p}\right)$$
and hence
$$\|\varphi(x)\|\leq m\|f(x,0)\|^{1/p},$$
where $m>0$ independent constant.
\end{proof}

\section{$P$-order tangent cone theorem for singular inclusions. Generalization of Lusternik theorem}

Consider the following mapping
\begin{equation}\label{eq20}
  F(x):=f(x)+N_C(x)
\end{equation}
and the generalized equation
\begin{equation}\label{eq21}
 0\in  F(x)
\end{equation}
where $f:X\rightarrow X^{\ast}$ and sufficiently smooth, $X$ -- Banach space, $C$ is a nonempty closed convex set in $X$ and $N_C(x)$ is defined in \eqref{conv}. Let $x_0$ be the solution to inclusion \eqref{eq21}, i.e. $0\in  F(x_0)$ and introduce the tangent cone to the set
$$M_F(x_0):=\left\{z\in X: 0\in f(z)+N_C(z)\right\}$$
at the point $x_0$.
\begin{defn}
We say that $h$ belongs to the tangent cone $TM_F(x_0)$ of the set $M_F(x_0)$ at the point $x_0$ if $\forall t\in [0,\varepsilon),$ where $\varepsilon>0$ sufficiently small, there exists a mapping $w:X\rightarrow X$ such that
$$0\in f(x_0+th+w(th))+N_C(x_0+th+w(th))$$
and $\|w(th)\|=o(t)$.
\end{defn}
It is enough to consider the completely degenerate case up to the order $p$, i.e. the case where $f^{(k)}(x_0)=0,$ $k=1,\ldots,p-1$, $p\geq 2$.

For any $h\in X$ we define a set-valued mapping $L_h:X\rightrightarrows X^{\ast},$
\begin{equation}\label{eq22}
  L_h(x):=\frac{1}{(p-1)!} f^{(p)}(x_0)[h]^{p-1}[h+x]+N_C(x_0+h+x).
\end{equation}
We can describe the tangent cone $TM_F(x_0)$ by means of the following theorem, which generalizes Lusternik theorem for singular inclusions.
\begin{thm}\label{th3}
Let $f:X\rightarrow X^{\ast}$, $f\in \mathcal{C}^{p+1}(X)$, $f$ be completely degenerate at $x_0$ up to the order $p$ and $0\in F(x_0)$. Assume moreover, that $\bar{h}\in \Ker^{p}L_h(0)$, i.e.
$$0\in L_{t\bar{h}}(0)\Leftrightarrow 0\in  f^{(p)}(x_0)[t\bar{h}]^{p}+N_C(x_0+t\bar{h}) \quad \forall t\in [0,\varepsilon)$$
where $\varepsilon>0$ sufficiently small and for the mapping $F(x)$ strong $p$-regularity condition 
holds along $\bar{h}$ at the point $x_0$, that is
\begin{equation}\label{eq23}
  H\left(L^{-1}_{t\bar{h}}(y_1),L^{-1}_{t\bar{h}}(y_2)\right)\leq \frac{c}{t^{p-1}}\|y_1-y_2\|
\end{equation}
where $y_1,y_2\in  X^{\ast}$ and $t\in [0,\varepsilon)$.

Then
\begin{equation}\label{eq24}
  \bar{h}\in TM_F(x_0).
\end{equation}
\end{thm}
\begin{proof}
For the sake of simplicity we assume that $x_0=0$. Let us define

$ r(t\bar{h}+x):=\frac{1}{(p-1)!}f^{(p)}(0)[t\bar{h}]^{p-1}[t\bar{h}+x]-f(t\bar{h}+x)$.

Then
\begin{equation}\label{eq25}
  0\in L_{t\bar{h}}^{-1}\left(\frac{1}{(p-1)!}f^{(p)}(0)[t\bar{h}]^{p}+N_C(t\bar{h})\right).
\end{equation}
Consider the following mapping
\begin{equation}\label{eq26}
  \Phi(x):=L_{t\bar{h}}^{-1}(r(t\bar{h}+x)).
\end{equation}
We show that there exists $w(t\bar{h})$ such that
$w(t\bar{h})\in \Phi(w(t\bar{h}))$ and \\ $\|w(t\bar{h})\|=o(t)$
or, in other words
$
  0\in f(t\bar{h}+w(t\bar{h}))+N_C(t\bar{h}+w(t\bar{h}).
$
From the condition \eqref{eq23} for $\|y_1\|\leq \alpha t$, $\|y_2\|\leq \alpha t$ where $\alpha>0$ is sufficiently small we have
\begin{enumerate}
\item[1$^{\circ}$] $H(\Phi(y_1),\Phi(y_2))= H\left(L_{t\bar{h}}^{-1}\left(r(t\bar{h}+y_1)\right),L_{t\bar{h}}^{-1}\left(r(t\bar{h}+y_2)\right)\right)\leq$\\
    $\leq\frac{c}{t^{p-1}}\left\|\frac{1}{(p-1)!}f^{(p)}(0)[t\bar{h}]^{p-1}(y_1-y_2)-
      f(t\bar{h}+y_1)+f(t\bar{h}+y_2)\right\|\leq$\\
 $\leq \delta(t)\|y_1-y_2\|,$ where $\delta(t)\rightarrow 0$ while $t\rightarrow 0$.
 \vskip6pt

\item[2$^{\circ}$] $H(\Phi(0),0)\leq$\\
$\leq H\left(L_{t\bar{h}}^{-1}\left(r(t\bar{h})\right),
     L_{t\bar{h}}^{-1}\left(\frac{1}{(p-1)!}f^{(p)}(0)[t\bar{h}]^{p}+N_C(t\bar{h})\right)\right)\leq$\\
     $\leq\frac{c}{t^{p-1}}\left\|r(t\bar{h})-\frac{1}{(p-1)!}f^{(p)}(0)[t\bar{h}]^{p}+N_C(t\bar{h})\right\|\leq$\\
      $\leq\frac{c}{t^{p-1}}\left\|\frac{1}{(p-1)!}f^{(p)}(0)[t\bar{h}]^{p}-
    f(t\bar{h})+0\right\|\leq c_1 t^2.$
\end{enumerate}

It means that all conditions of CMP are fulfilled and hence
there exists $w(t\bar{h})\in \Phi(t\bar{h})$ and it follows that $w(t\bar{h})\in L_{t\bar{h}}^{-1}\left(r(t\bar{h}+w(t\bar{h}))\right)$,
or in other words
$ 0\in f(t\bar{h}+w(t\bar{h}))+N_C(t\bar{h}+w(t\bar{h})$ $\forall t\in [0,\varepsilon)$ and $\|w(t\bar{h})\|=o(t)$, i.e.
$\bar{h}\in TM_F(x_0)$.
\end{proof}

\vskip3pt

\noindent
\textbf{Example 3.}

Let $F(x)=f(x)+N_C(x)$, where $f(x)=(f_1(x),f_2(x))^T$, $x\in \mathbb{R}^2$,\\
 $f_1(x)=x_2^2-x_1^2$, $f_2(x)=x_1 x_2$, $x\geq 0$, $C=\mathbb{R}_+^2$, $x_0=0$. Here $p=2$,
 $N_{\mathbb{R}_+^2}(0)=\left\{z\in \mathbb{R}^2: \langle z,\xi \rangle\leq 0, \forall \xi\in \mathbb{R}_+^2 \right\}$
 and generalized equation is
 $$0\in \left(
                \begin{array}{c}
                  x_2^2-x_1^2 \\
                  x_1 x_2\\
                \end{array}
              \right)+N_{\mathbb{R}_+^2}\left(
                \begin{array}{c}
                  x_1 \\
                  x_2\\
                \end{array}
              \right). $$
 Taking $\bar{h}=(\bar{h}_{x_1},\bar{h}_{x_2})^T$ where $\bar{h}_{x_1}=0$ $\bar{h}_{x_2}=1$ we have
 $N_{\mathbb{R}_+^2}\left(
                \begin{smallmatrix}
                  0 \\
                  1\\
                \end{smallmatrix}
              \right)= \alpha \left(
                \begin{array}{c}
                  -1 \\
                  0\\
                \end{array}
              \right)$, $\alpha>0$ and $\bar{h}\in \Ker^2L_h(0)$ if
$\left(
                \begin{array}{c}
                  \bar{h}_{x_1} \\
                  0\\
                \end{array}
              \right)-  \left(
                \begin{array}{c}
                 \alpha\\
                  0\\
                \end{array}
              \right)=0$, i.e. $\bar{h}_{x_1}^2=\alpha$ and $\bar{h}_{x_1}=\pm\sqrt{\alpha}$ and
              all assumptions of Theorem \ref{th3} are fulfilled. It means that $\bar{h}=(0,1)^T\in TM_F(0).$


\end{document}